\documentclass{amsart}

\usepackage{amssymb}
\usepackage{graphicx}
\usepackage{MnSymbol}

\usepackage{enumerate}

\usepackage{amsmath,amscd}

\usepackage{amsthm}

\title{Infinite simple groups with no proper abnormal subgroups}

\author{Samuel M. Corson}
\address{E. T. S. I. I. Universidad Polit\'{e}cnica de Madrid, Jos\'{e} Guti\'{e}rrez Abascal 2, 28006 Madrid, Spain}
\email{sammyc973@gmail.com}

\theoremstyle{definition}\newtheorem{theorem}{Theorem}

\theoremstyle{definition}
\theoremstyle{definition}

\theoremstyle{definition}

\theoremstyle{definition}
\theoremstyle{definition}\newtheorem{proposition}[theorem]{Proposition}
\theoremstyle{definition}
\theoremstyle{definition}
\theoremstyle{definition}
\theoremstyle{definition}\newtheorem{remark}[theorem]{Remark}
\theoremstyle{definition}
\theoremstyle{definition}\newtheorem{lemma}[theorem]{Lemma}
\theoremstyle{definition}
\theoremstyle{definition}
\theoremstyle{definition}
\theoremstyle{definition}
\theoremstyle{definition}
\theoremstyle{definition}

\newcommand{\fin}{\operatorname{fin}}

\newcommand{\Po}{\mathcal{P}}

\begin{document}



\thanks{The author is partially supported by RYC2023-045493-I.}

\begin{abstract}
Utilizing an embedding theorem of Obraztsov we construct groups as described in the title.  This provides an affirmative answer to a problem of D. O. Revin.  The constructed groups also provide a negative answer to a question highlighted by Kurdachenko, Russo, and Vincenzi.
\end{abstract}

\maketitle

\begin{section}{Introduction}

Recall that a subgroup $H$ of a group $G$ is \emph{abnormal} if $g \in \langle H \cup g^{-1}Hg \rangle$ for every $g \in G$.  Evidently an abnormal subgroup is very different from a normal subgroup; the whole group $G$ is the only subgroup of $G$ which is both normal and abnormal.  Such subgroups were studied by Philip Hall already in the 1930s \cite{Hall}, although the expression ``abnormal subgroup'' was coined much later by Carter \cite{Carter}.

Revin has asked whether there exists an infinite simple group that does not have proper abnormal subgroups \cite[Problem 19.74]{KhMaz}.  We show that there is such a group.

\begin{theorem}\label{thebigmain}  There exists an infinite group $H$ which

\begin{enumerate}

\item is simple;

\item is torsion-free;

\item has no maximal subgroup; and

\item has no proper abnormal subgroup.

\end{enumerate}

\noindent Furthermore, we can have $H$ satisfy either

\begin{enumerate}

\item[(a)] $|H| = \aleph_0$; or

\item[(b)] $|H| = \aleph_1$ and each proper subgroup of $H$ is countable.

\end{enumerate}
\end{theorem}

Part of the interest of such a group $H$ is that the only (ab)normal subgroups are the obvious ones.  It was an open question whether having no abnormal subgroups implies that a group is locally nilpotent \cite{KRV}.  The groups provided by Theorem \ref{thebigmain} are not locally nilpotent (see Remark \ref{notlocallynilpotent}), thus a negative answer to this question is established.

The group is constructed as a union of an increasing well-ordered chain of groups (the chain is of countable or uncountable length in the case of (a) or of (b), respectively).  Arbitrarily high in the chain there are simple groups which are normal in the next larger group, and any proper subgroup of $H$ includes into a group in the chain.  The proof follows a modification of a construction of Obraztsov \cite{Ob0} (see also \cite[Theorems 35.2, 35.3]{Ol}), using a consequence of his newer embedding theorem \cite{Ob} instead of his earlier \cite[Theorem 35.1]{Ol}.

We give some background and a lemma for induction in Section \ref{Prep} and then construct the groups in Section \ref{theconstuction}.

\end{section}

\begin{section}{Background Results}\label{Prep}

If $X$ is a set, let $\Po'(X) = \{Y \subseteq X \mid Y \neq \emptyset\}$ and $\Po_{\fin}'(X) = \{Y \in \Po'(X)| Y\text{ is finite}\}$.  We use the convention that a set is countable if it is either finite or countably infinite.  If $\{G_i\}_{i \in I}$ is a collection of groups let $\Omega^1$ denote the free amalgam of groups; that is, the union $\Omega^1 = \bigcup_{i \in I} G_i$ where $G_i \cap G_j = \{1\}$ whenever $i \neq j$.  Letting $\Omega = \Omega^1 \setminus \{1\}$ we say a function $f: \Po'(\Omega) \rightarrow \Po'(\Omega)$ is \emph{generating} if the following hold:

\begin{enumerate}

\item if $X \subseteq G_i$ for some $i \in I$ then $f(X) = \langle X \rangle \setminus \{1\}$;

\item if $X \in \Po_{\fin}'(\Omega)$ and $X \not\subseteq G_i$ for each $i \in I$ then $f(X) = Y$ is a countable subset of $\Omega$ such that $X \subseteq Y$ and for $Z \in \Po_{\fin}'(Y)$ we have $f(Z) \subseteq Y$;

\item if $X \in \Po'(\Omega)$ is infinite and $X \not\subseteq G_i$ for each $i \in I$ then $f(X) = \bigcup_{Y \in \Po_{\fin}'(X)} f(Y)$.
\end{enumerate}

Suppose $\phi: *_{i \in I} G_i \rightarrow G$ is a surjective homomorphism such that $\phi \upharpoonright \Omega^1$ is injective.  We shall in this case identify each $G_i$ as a subgroup of $G$.  We can write each element $g \in G$ as a \emph{reduced} word $W =_G g$ in the letters $\Omega$ (that is, if $W \equiv g_1g_2 \ldots g_k$ and if $g_s \in G_{i_s}$ then $i_s \neq i_{s + 1}$).  We say such a word $W$ is \emph{minimal} if whenever $g =_G U$, with $U$ a word in the letters $\Omega$, $U$ is at least as long as $W$.  If, in addition, $f: \Po'(\Omega) \rightarrow \Po'(\Omega)$ is generating we let $\mathcal{W}$ be the set of nonempty reduced words over the alphabet $\Omega$ and define a function $F: \Po'(\mathcal{W}) \rightarrow \Po'(\Omega)$ as follows:  for a nonempty subset $\mathcal{U} \subseteq \mathcal{W}$ we set $Y_{\mathcal{U}}$ to be the set of letters which occur in some word in $\mathcal{U}$ and let $F(\mathcal{U}) = f(Y_{\mathcal{U}})$.

For the following, see \cite[Theorem A]{Ob} (see \cite[Lemma 3]{Ob} for part (v)).

\begin{proposition}\label{beautifulOb}
Suppose the following.

\begin{enumerate}
\item[$(\dagger)_1$] $\{G_i\}_{i \in I}$ is a collection of nontrivial torsion-free groups, $\Omega^1$ is the free amalgam of these groups, and $i_0 \in I$;

\item[$(\dagger)_2$] $I_1 := I \setminus \{i_0\}$ has cardinality at least $2$;

\item[$(\dagger)_3$] $\Omega_1^1$ is the free amalgam of the collection $\{G_j\}_{j \in I_1}$;

\item[$(\dagger)_4$] $f: \Po'(\Omega) \rightarrow \Po'(\Omega)$ is generating, and if $Y \not\subseteq G_i$ for each $i \in I$ then $f(Y) \cap \Omega_1^1 \neq \emptyset$.

\end{enumerate}

\noindent Then there exists a torsion-free group $G$ and homomorphism $\phi: *_{i \in I} G_i \rightarrow G$ such that

\begin{enumerate}

\item[(i)] $\phi$ is surjective;

\item[(ii)] $\phi \upharpoonright \Omega^1$ is injective;

\item[(iii)] $\phi(\Omega_1^1)$ is a subset of a normal simple subgroup $L$ of $G$ such that $G$ is the semidirect product $G = L \rtimes G_{i_0}$;

\item[(iv)] if $M$ is a subgroup of $G$ then 

\begin{enumerate}

\item[$(o)_1$] $M$ is cyclic; or

\item[$(o)_2$] $M \cap L = \{1\}$; or

\item[$(o)_3$] $M$ is conjugate in $G$ to a subgroup $M_0 \leq G$ having a normal subgroup $L_0 \unlhd M_0$, such that

\begin{enumerate}

\item[$(t)_1$] $M_0/L_0$ is isomorphic to a subgroup of $G_{i_0}$;

\item[$(t)_2$] if we select for each $g \in L_0 \setminus \{1\}$ a minimal word $W_g =_G g$ over the elements of $\Omega$ and let $X = F(\{W_g\}_{g \in L_0 \setminus \{1\}})$ (or $X = \emptyset$ in case $L_0 = \{1\}$) then the inclusion $L_0 \leq \langle X \rangle \cap L$ holds and moreover if $X \not\subseteq G_i$ for each $i \in I$ then $L_0 = \langle X \rangle \cap L$ and $M_0 \leq \langle X \rangle$;

\end{enumerate}

\end{enumerate}

\item[(v)] if $i \in I$ and $g \in G \setminus G_i$ then for all $j \in I$ we have $G_j \cap g^{-1}G_ig = \{1\}$.

\end{enumerate}

\end{proposition}

From this powerful result we derive the following lemma which will be used in our inductive construction.

\begin{lemma}\label{inductivestep}
If $K_0$ is a countably infinite torsion-free group, then there exist groups $K_1$ and $K_2$ such that

\begin{enumerate}

\item[$(*)_1$] $K_2$ is countable and torsion-free;

\item[$(*)_2$] $K_i$ is a proper subgroup of $K_{i + 1}$ for $i = 0, 1$;

\item[$(*)_3$] $K_1$ is a simple normal subgroup of $K_2$; and

\item[$(*)_4$] if $g \in K_0 \setminus \{1\}$ and $h \in K_2 \setminus K_0$ then $\langle g, h \rangle \geq K_1$.

\end{enumerate}

\end{lemma}

\begin{proof}
Assume $K_0$ is a countably infinite torsion-free group.  We consider the collection of groups $\{G_0, G_1, G_2\}$ where $G_0 = K_0$, and both $G_1$ and $G_2$ are isomorphic to $\mathbb{Z}$.  As is done in \cite[Theorem D]{Ob} we let $f: \Po'(\Omega) \rightarrow \Po'(\Omega)$ be given by

\[
f(X) = \left\{
\begin{array}{ll}
\langle X \rangle \setminus \{1\}
                                            & \text{if } X \subseteq G_i \text{ for some }i \in I, \\
\Omega                                       & \text{otherwise}
\end{array}
\right.
\]

\noindent and note that $f$ is generating.  Now invoke Proposition \ref{beautifulOb} with $i_0 = 2$.  Let $K_2 = G$ and $K_1 = L$.  Properties $(*)_1$ - $(*)_3$ are clear.

We check property $(*)_4$.  Suppose $g \in K_0 \setminus \{1\}$ and $h \in K_2 \setminus K_0$ and let $M = \langle g, h\rangle$.  By condition (v) we have $\langle g \rangle \cap h^{-1}\langle g \rangle h \leq K_0 \cap h^{-1} K_0 h = \{1\}$, and so $M$ is not abelian and therefore not cyclic.  Also, $M \cap L \geq \langle g \rangle \neq \{1\}$.  Thus by property (iv) there exists some $h_0 \in K_2$  and $L_0$ such that $L_0 \unlhd M_0 = h_0^{-1} M h_0$ as in $(o)_3$.  Since $M_0/L_0$ is isomorphic to a subgroup of $G_2 \simeq \mathbb{Z}$, and $M_0 \simeq M$ is not abelian, we know that $L_0$ is not trivial.

We next note that $M_0 \setminus (G_0 \cup G_1 \cup G_2) \neq \emptyset$, which we check by cases.  If $h_0 \notin G_0$ we have $h_0^{-1}gh_0 \notin G_0 \cup G_1 \cup G_2$ (by condition (v)).  On the other hand, if $h_0 \in G_0$ we get $h_0^{-1}hh_0 \notin G_0$.  If $h_0^{-1}hh_0 \notin G_1 \cup G_2$ then we are done, but if say $h_0^{-1}hh_0 \in G_i$ then $(h_0^{-1}gh_0)^{-1}(h_0^{-1}hh_0)(h_0^{-1}gh_0) \notin G_0 \cup G_1 \cup G_2$ (again by (v)).  Thus $M_0 \setminus (G_0 \cup G_1 \cup G_2) \neq \emptyset$ holds in any case.

Next, we notice that $L_0 \not\subseteq G_i$ for $i = 0, 1, 2$.  We have already seen that $L_0$ is nontrivial, say $h_1 \in L_0 \setminus \{1\}$.  By way of contradiction if $L_0 \leq G_i$ then take $h_2 \in M_0 \setminus (G_0 \cup G_1 \cup G_2)$ and note that since $L_0$ is normal in $M_0$ we have $$h_1 \in L_0 \cap h_2^{-1}L_0h_2 \leq G_i \cap h_2^{-1} G_i  h_2= \{1\}$$ where the last equality is by condition (v), a contradiction.

For each $g_1 \in L_0 \setminus \{1\}$ let $W_{g_1}$ be a word of minimal length in the letters $\Omega$ representing $g_1$.  As $L_0 \not\subseteq G_i$ for $i = 0, 1, 2$ we know there are letters used in the words $W_{g_1}$ which come from distinct $G_i$, and therefore $F(\{W_{g_1}\}_{g_1 \in L_0 \setminus \{1\}}) = X = \Omega$.  Thus $L_0 = \langle X \rangle \cap L = L$ by $(t)_2$, and $M_0 \geq L_0 = L$.  Since $L$ is normal in $G$, we see that $M \geq L$, and the proof of $(*)_4$  is complete.

\end{proof}

\end{section}

\begin{section}{Construction and verification} \label{theconstuction}

We will make use of ordinal numbers (see \cite{J} for greater details). We follow the convention that an ordinal is the set of all ordinals that are below it (e.g. $4 = \{0, 1, 2 , 3\}$).  Thus $\alpha \in \beta$ if and only if the ordinal $\alpha$ is strictly below the ordinal $\beta$, and if $Z$ is a set of ordinals then the union $\bigcup Z$ is an ordinal (equal to the supremum of $Z$).  An ordinal $\alpha$ is \emph{limit} if $\alpha$ does not contain a maximal element (equivalently, $\alpha = \bigcup \alpha$).  A subset $Z \subseteq \alpha$ is \emph{unbounded} in limit ordinal $\alpha$ if $\bigcup Z = \alpha$.  Recall that each ordinal $\alpha$ can be written uniquely as $\alpha = \beta + n$, where $\beta$ is a limit ordinal and $n$ is a natural number.  Thus we consider $\alpha$ to be even (respectively odd ) if $n$ is even (resp. odd).

Let $\omega_1$ denote the first uncountable ordinal.  We build an increasing chain $(H_{\alpha}: \alpha \leq \omega_1)$ of infinite torsion-free groups.  Let $H_0$ be isomorphic to $\mathbb{Z}$.  Suppose that we have constructed groups $H_{\alpha}$ for all $\alpha < \beta \leq \omega_1$ which are countably infinite and torsion-free.  If $0 < \beta \leq \omega_1$ is a limit ordinal then let $H_{\beta} = \bigcup_{\alpha < \beta} H_{\alpha}$.

If $\beta < \omega_1$ is not limit and $\beta$ is odd then write $\beta = \alpha + 1$ (so $\alpha$ is even).  By Lemma \ref{inductivestep} take $H_{\alpha + 1}$ and $H_{\alpha + 2}$ such that 

\begin{itemize}

\item $H_{\alpha + 2}$ is countable and torsion-free;

\item $H_{\alpha + i}$ is a proper subgroup of $H_{\alpha + i + 1}$ for $i = 0, 1$;

\item $H_{\alpha + 1}$ is a simple normal subgroup of $H_{\alpha + 2}$;

\item if $g \in H_{\alpha} \setminus \{1\}$ and $h \in H_{\alpha + 2} \setminus H_{\alpha}$ then $\langle g, h \rangle \geq H_{\alpha + 1}$.
\end{itemize}

\noindent Note that this also covers the case when $\beta$ is even and not limit, so the chain $(H_{\alpha}: \alpha \leq \omega_1)$ is defined.

In what follows let $\mathbb{E} = \{\alpha < \omega_1 \mid \alpha \text{ is even}\}$ and $\mathbb{O} =\omega_1 \setminus \mathbb{E}$.

\begin{lemma}\label{tightbound}  Suppose $0 <\beta \leq \omega_1$ is a limit ordinal, $K \leq H_{\omega_1}$, and the set $$\{\alpha < \beta \mid K \cap (H_{\alpha + 1} \setminus H_{\alpha})\neq \emptyset\}$$ is unbounded in $\beta$.  Then $H_{\beta} \leq K$.

\noindent In particular, if $K$ is a proper subgroup of $H_{\beta}$ then $K \leq H_{\gamma}$ for some $\gamma < \beta$.
\end{lemma}

\begin{proof}
As $H_{\beta} = \bigcup_{\alpha < \beta} H_{\alpha}$ it is sufficient to show that for arbitrary $\alpha_0 < \beta$ we have $H_{\alpha_0} \leq K$.  For a fixed $\alpha_0$ we may select, by assumption, $\alpha_0 < \alpha_1 < \beta$ and $h_1 \in K \cap (H_{\alpha_1 + 1} \setminus H_{\alpha_1})$.  Select $h_2 \in K \cap (H_{\alpha_2 + 1} \setminus H_{\alpha_2})$ where $\alpha_1 + 2 \leq \alpha_2 < \beta$.  Take $\alpha_2' = \max\{\alpha \leq \alpha_2 \mid \alpha \in \mathbb{E}\}$, so either $\alpha_2' = \alpha_2$ or $\alpha_2' + 1 = \alpha_2$.  Now $h_1 \in H_{\alpha_2'} \setminus \{1\}$ and $h_2 \in H_{\alpha_2' + 2} \setminus H_{\alpha_2'}$, so $K \geq \langle h_1, h_2 \rangle \geq H_{\alpha_2' + 1} \geq H_{\alpha_0}$.

Supposing $K$ is a proper subgroup of $H_{\beta}$, the set $Z = \{\alpha < \beta \mid K \cap (H_{\alpha + 1} \setminus H_{\alpha})\neq \emptyset\}$ is bounded in $\beta$, say $\gamma' = \bigcup Z < \beta$, and we have $K \leq H_{\gamma' + 1}$ and $\gamma = \gamma' + 1 < \beta$.
\end{proof}

\begin{lemma}\label{verification}  If $0 < \beta \leq \omega_1$ is a limit ordinal then $H_{\beta}$ is an infinite group satisfying properties (1) - (4) in the statement of Theorem \ref{thebigmain}.
\end{lemma}

\begin{proof}
It is clear that $H_{\beta}$ is infinite since $H_{\beta} \geq H_0$.  We know that $H_{\beta}$ is (1) simple since $H_{\beta} = \bigcup_{\alpha < \beta} H_{\alpha} = \bigcup_{\alpha \in \beta \cap \mathbb{O}} H_{\alpha}$ and $H_{\alpha}$ is simple for $\alpha \in \beta \cap \mathbb{O}$.  That $H_{\beta}$ is (2) torsion-free is clear (by induction on $\alpha < \beta$).  If $K$ is a proper subgroup of $H_{\beta}$ then by Lemma \ref{tightbound} there is some $\gamma < \beta$ for which $K \leq H_{\gamma}$, and so $K$ is not maximal as $K \leq H_{\gamma} < H_{\gamma + 1} < H_{\beta}$ and property (3) holds.  For property (4), take $K$ to be a proper subgroup of $H_{\beta}$, select $\gamma < \beta$ for which $K \leq H_{\gamma}$ (without loss of generality $\gamma \in \mathbb{E}$) and pick $h \in H_{\gamma + 2} \setminus H_{\gamma + 1}$.  Since $H_{\gamma + 1} \unlhd H_{\gamma + 2}$ we see that $\langle K \cup h^{-1} K h\rangle \leq H_{\gamma + 1}$ and so $h \notin \langle K \cup h^{-1} K h\rangle$ and $K$ is abnormal in $H_{\beta}$.
\end{proof}

For (a) in Theorem \ref{thebigmain} we let $H = H_{\omega}$ (where $\omega$ is the smallest infinite ordinal).  For (b) we let $H = H_{\omega_1}$.  It is clear in this case that $|H| = \aleph_1$, and that every proper subgroup of $H$ is countable is immediate from Lemma \ref{tightbound} and the fact that $H_{\gamma}$ is countable for $\gamma < \omega_1$.

\begin{remark}\label{notlocallynilpotent}  It is easy to see that the constructed group $H$, either in the countable or in the uncountable case, is not locally nilpotent.  For example, taking $g \in H_0 \setminus \{1\}$ and $h \in H_2 \setminus H_0$ we know $H_1 \leq \langle g, h \rangle \leq H$.  The infinite simple group $H_1$ is not nilpotent, so neither is the subgroup $\langle g, h \rangle$.

\end{remark}

\end{section}


\begin{thebibliography}{a}

\bibitem{Carter}  R. W. Carter, \emph{Nilpotent self-normalizing subgroups of soluble groups}, Math. Z. 75 (1961), 136-139.

\bibitem{Hall}  P. Hall, \emph{On the system normalizers of soluble groups}, Proc. London Math. Soc. 43 (1937), 507-528.

\bibitem{J} T. Jech, Set Theory: The Third Millenium Edition, Revised and Expanded, Springer (2000).

\bibitem{KhMaz}  E. I. Khukhro and V. D. Mazurov, editors, \emph{Unsolved problems in group theory, the Kourovka notebook}, 20th edition (2022). Sobolev Institute of Mathematics, Russian Academy of Sciences, Siberian Branch, Novosibirsk.

\bibitem{KRV}  L. A. Kurdachenko, A. Russo, G. Vincenzi, \emph{Groups without proper abnormal subgroups}, J. Group Theory 9 (2006), 507-518.

\bibitem{Ob0} V. N. Obraztsov, \emph{An embedding theorem for groups and its corollaries}, Mat. Sb. 180 (1989), 529-541.

\bibitem{Ob} V. N. Obraztsov, \emph{A new embedding scheme for groups and some applications}, J. Austral. Math. Soc. (Series A), 61 (1996), 267-288.

\bibitem{Ol} A. Yu. Ol'shanskii, Geometry of Defining Relations in Groups, Kluwer Academic Publisher (1991).


\end{thebibliography}
\end{document}